\documentclass{amsart}
\usepackage{amssymb,
enumitem,
mathrsfs,
mathtools,
tikz,
upgreek,
verbatim,
hyphenat,
url
}
\usepackage[T1]{fontenc}
\usepackage[shadow]{todonotes}
\usepackage{tikz-cd}
 
\usepackage{newpxtext,newpxmath}
\usepackage{subfig}
\usetikzlibrary{calc}

\usetikzlibrary{shapes.misc, positioning}

\usepackage[colorlinks=true,linkcolor={black},citecolor={blue}]{hyperref}%

\newcommand{\LO}{\mathrm{LO}}

\DeclareMathOperator{\Aut}{Aut}
\DeclareMathOperator{\BS}{BS}

\DeclareMathOperator{\Homeo}{Homeo}
\DeclareMathOperator{\Orb}{Orb}
\DeclareMathOperator{\Stab}{Stab}

\newenvironment{enumerate-(a)}{\begin{enumerate}[label={\upshape (\alph*)}, leftmargin=2pc]}{\end{enumerate}}
\newenvironment{enumerate-(a)-r}{\begin{enumerate}[label={\upshape (\alph*)}, leftmargin=2pc,resume]}{\end{enumerate}}
\newenvironment{enumerate-(a)-5}{\begin{enumerate}[label={\upshape (\alph*)}, leftmargin=2pc,start=5]}{\end{enumerate}}
\newenvironment{enumerate-(A)}{\begin{enumerate}[label={\upshape (\Alph*)}, leftmargin=2pc]}{\end{enumerate}}
\newenvironment{enumerate-(A)-r}{\begin{enumerate}[label={\upshape (\Alph*)}, leftmargin=2pc,resume]}{\end{enumerate}}
\newenvironment{enumerate-(i)}{\begin{enumerate}[label={\upshape (\roman*)}, leftmargin=2pc]}{\end{enumerate}}
\newenvironment{enumerate-(i)-r}{\begin{enumerate}[label={\upshape (\roman*)}, leftmargin=2pc,resume]}{\end{enumerate}}
\newenvironment{enumerate-(I)}{\begin{enumerate}[label={\upshape (\Roman*)}, leftmargin=2pc]}{\end{enumerate}}
\newenvironment{enumerate-(I)-r}{\begin{enumerate}[label={\upshape (\Roman*)}, leftmargin=2pc,resume]}{\end{enumerate}}
\newenvironment{enumerate-(1)}{\begin{enumerate}[label={\upshape (\arabic*)}, leftmargin=2pc]}{\end{enumerate}}
\newenvironment{enumerate-(1)-r}{\begin{enumerate}[label={\upshape (\arabic*)}, leftmargin=2pc,resume]}{\end{enumerate}}

\newtheorem{theorem}{Theorem}[section]
\newtheorem{lemma}[theorem]{Lemma}
\newtheorem{corollary}[theorem]{Corollary}
\newtheorem{proposition}[theorem]{Proposition}
\newtheorem{question}[theorem]{Question}

\theoremstyle{definition}

\theoremstyle{remark}
\newtheorem{remark}[theorem]{Remark}

\newtheorem{innercustomgeneric}{\customgenericname}
\providecommand{\customgenericname}{}
\newcommand{\newcustomtheorem}[2]{%
  \newenvironment{#1}[1]
  {%
   \renewcommand\customgenericname{#2}%
   \renewcommand\theinnercustomgeneric{##1}%
   \innercustomgeneric
  }
  {\endinnercustomgeneric}
}

\newcustomtheorem{customthm}{Theorem}
\newcustomtheorem{customlemma}{Lemma}

\begin{document}

\title[Condensation and left-orderable groups]{Condensation and left-orderable groups}

\date{}
\author[F.~Calderoni]{Filippo Calderoni}

\address{Department of Mathematics, Rutgers University, Hill Center for the Mathematical Sciences, 110 Frelinghuysen Rd., Piscataway, NJ 08854-8019} \email{filippo.calderoni@rutgers.edu}

\author[A.~Clay]{Adam Clay}
\address{Department of Mathematics, 420 Machray Hall, University of Manitoba, Winnipeg, MB, R3T 2N2, Canada} \email{Adam.Clay@umanitoba.ca}

 \subjclass[2020]{Primary: 03E15, 06F15, 20F60.}

 \thanks{We are grateful to the anonymous referee for pointing out Proposition~\ref{prop : convex} and sharing the proof. The first author would like to thank Justin Moore for asking about the Borel complexity of the conjugacy action of Thompson's group \(F\) on its space of left-orderings. This prompted the investigation of Section~\ref{sec : Thompson}. Calderoni's research was partially supported by the NSF grant DMS--2348819. Clay's research was partially supported by NSERC grant RGPIN-2020-05343.}

\begin{abstract}
We discuss condensed left-orderings and develop new techniques to show that the conjugacy relation on the space of left-orderings is not smooth.
These techniques apply to the solvable Baumslag Solitar groups \(\BS(1,n)\) and to Thompson's group \(F\).
\end{abstract}

\maketitle


\section{Introduction}

The study of definable quotients of Polish spaces is one of the main themes in modern descriptive set theory, with the primary goal being to understand the Borel structure of Polish spaces modulo analytic equivalence relations. A fundamental question is whether the quotient space, equipped with the quotient Borel structure, is standard.  The first trace of such an analysis dates back to the work of Glimm~\cite{Gli61} and Effros~\cite{Eff}.

If \(G\) is a countable group acting continuously on the Polish space \(X\), we denote by \(X/G\) the space of orbits, endowed with the quotient Borel structure. In this case \(X/G\) is standard if and only if the orbit equivalence relation on \(X\) induced by the \(G\)-action is \emph{smooth}. That is, if and only if there is a Borel map \(\theta\colon X\to \mathbb{R}\) such that \(x_1, x_2\) lie in the same \(G\)-orbit if and only if \(\theta(x_1) = \theta(x_2)\).  It is owing to this characterization that for the remainder of the manuscript, we assume that all groups are countable unless otherwise indicated.

Following this trend, a question posed by Deroin, Navas, and Rivas~\cite{DerNavRiv} raised the problem of whether the space of left-orderings $\LO(G)$ of a left-orderable group \(G\) modulo the conjugacy \(G\)-action is always standard, or equivalently, whether or not the orbit equivalence relation is always smooth. Using descriptive set theory to demonstrate non-smoothness, the authors of this manuscript showed that \(\LO(G)/G\) is not a standard Borel space in many cases; for example, when \(G\) is a  non-abelian free group or a free product of left-orderable groups~\cite{CalCla22, CalCla23}.

Denote the equivalence relation induced by the conjugacy action of a  group $G$ on its space of left-orderings by $E_\mathsf{lo}(G)$, and let $\BS(1,n)$ denote the Baumslag-Solitar group.  In this manuscript, we show:

\begin{theorem}
\label{thm : BS}
\begin{enumerate-(1)}
    \item 
    \label{item : main 1}
    For all \(n\geq 2\), \(E_\mathsf{lo}(\BS(1, n))\) is not smooth.
    \item 
    \label{item : main 2}
    The conjugacy equivalence relation \(E_\mathsf{lo}(\BS(1, 2))\) is Borel bi-reducible with \(E_0\).
\end{enumerate-(1)}
\end{theorem}

The novelty of Theorem~\ref{thm : BS} is twofold.
Theorem~\ref{thm : BS}\ref{item : main 1} provides the first examples of left-orderable solvable groups \(G\) with non-standard quotient \(\LO(G)/G\).
Moreover, Theorem~\ref{thm : BS}\ref{item : main 2} shows the first example of a finitely generated group \(G\) for which \(E_\mathsf{lo}(G)\) is not smooth, yet hyperfinite.

Using similar techniques, we also  show how work of Navas implies that \(E_\mathsf{lo}(G)\) is not smooth whenever $\LO(G)$ contains isolated points, and are able to tackle Thompson's group $F$ in a similar manner.

\begin{theorem}
\label{thm : F}
For Thompson's group \(F\), the conjugacy relation  $E_{\mathrm{lo}}(F)$  is not smooth.
\end{theorem}

Central to our analysis is the idea of condensed left-orderings in \(\LO(G)\), which are orderings that can be approximated by their conjugates.  Their existence turns out to be equivalent to non-smoothness of $E_{\mathsf{lo}}(G)$ (Proposition \ref{prop : smooth acc points}), moreover, they can be detected by analyzing \(\LO((G))\), the free part of the conjugacy \(G\)-action on \(\LO(G)\). (See Proposition \ref{prop : free part}.)  

\section{Condensed points}

A \emph{Polish space} is a separable and completely metrizable topological space.
For a Polish space \(X\) we denote by \(F(X)\) the Effros standard Borel space of closed subsets of \(X\). The standard Borel structure on \(F(X)\) is generated by the sets \[F_U=\{F\in F(X) \mid F\cap U\neq \emptyset\}\] for all open \(U\subseteq X\).

An equivalence relation \(E\) on the Polish space \(X\) is \emph{Borel} if \(E \subseteq X\times X\) is a Borel subset of \(X\times X\). Most of the Borel equivalence relations that we will consider in this paper arise from group actions as follows. Let \(G\) be a countable discrete group. Then a \emph{Polish \(G\)-space} is a Polish space \(X\) equipped with a continuous action \((g, x) \mapsto g \cdot x\) of \(G\) on \(X\). The corresponding orbit equivalence relation on \(X\), which we will denote by \(E^X_G\), is a Borel equivalence relation with countable classes. Let \(X\) be a fixed Polish \(G\)-space. For a subgroup \(H\leq G\), denote \(\Orb_H(x)\) the orbit of \(x\) under the induced \(H\)-action.
Whenever \(G=H\), we let \(\Orb(x)= \Orb_G(x)\).

Recall that a Borel equivalence relation \(E\) is \emph{smooth} if there exist a standard Borel space \(Y\) and a Borel map \(\theta\colon X \to Y\) such that
\[
x_1\mathbin{E} x_2 \iff \theta(x_1) = \theta(x_2).
\]

An equivalence relation on a Polish space is \emph{generically ergodic} if
every invariant set with the Baire property is meager or comeager. 
 Whenever \(X\) is a Polish \(G\)-space, the following are equivalent:
\begin{enumerate-(i)}
\item 
\(E_G^X\) is generically ergodic.
\item
There is \(x\in X\) such that \(\Orb(x)\) is dense in \(X\).
\end{enumerate-(i)}

Generic ergodicity is an obstruction to smoothness in many cases. In this manuscript, we will use the following fact:

\begin{proposition}[{E.g., see \cite[Corollary~3.5]{Hjo}}]
Suppose that \(G\) is a countable group and \(X\) is a Polish \(G\)-space with no isolated points.
If \(E^X_G\) is generically ergodic, then \(E^X_G\) is not smooth.
\end{proposition}

Following the terminology of Osin~\cite{Osi21APAL,Osi21}, we say that a point \(x\in X\) is \emph{condensed} if it is an accumulation point of \(\Orb(x)\).

The next proposition is essentially due to Osin~\cite[Proposition~2.7]{Osi21APAL}, who analyzed condensation in the Polish space of finitely generated marked groups. Since we could not find the proof in the literature, we give the proof of this general fact below.

\begin{proposition}
\label{prop : smooth acc points}
Suppose that \(G\) is a countable group and \(X\) is a Polish \(G\)-space. Then the following are equivalent:
\begin{enumerate-(1)}
\item
\label{item : 1}
\(E^X_G\) is smooth.
\item
\label{item : 2}
The are no condensed points in \(X\).
\end{enumerate-(1)}
\end{proposition}

\begin{proof}
Suppose that \(E^X_G\) is smooth and let \(x\) be any element of \(X\).
Consider the closed \(G\)-invariant set \(Y=\overline{\Orb(x)}\). If \(E^X_G\) is smooth, then \(E^Y_G\) is also smooth. As \(\Orb(x)\) is a dense \(G\)-orbit in \(Y\), the action \(G\curvearrowright Y\) is generically ergodic,  so there must be an isolated point in \( x_0 \in Y\). The point $x_0$ cannot be an element of $Y \setminus \Orb(x)$ since these points are non-isolated by definition, and so \(x_0 \in \Orb(x)\). Now as the $G$-action is continuous, every point of \(\Orb(x)\),  and in particular \(x\) itself, must be isolated in the subspace topology.  It follows that $x$ cannot be a condensed point.

On the other hand, suppose that no $x\in X$ is a condensed point.  Then for every $x\in X$, the subspace topology on $\Orb(x)$ is discrete, and since $\Orb(x)$ is countable it is therefore Polish.  By Alexandrov's theorem, $\Orb(x)$ must be a $G_{\delta}$ set for all $x \in X$. (See~\cite[Theorem~3.11]{Kec}.)  Further note that the saturation of an arbitrary open set $U \subset X$ is itself open, since the saturation can be written as a union of the sets $g \cdot U$ where $g \in G$, each of which is open since $G$ acts continuously.  This implies that the map \(X\to F(X), x\mapsto \overline{\Orb(x)}\) is Borel showing that \(E^X_G\) is Borel reducible to \(=_{F(X)}\). (E.g., see~\cite[Exercise~5.4.8]{Gao}.) 
\end{proof}




\section{The conjugacy relation on the spaces of left-orderings}

A group \(G\) is \emph{left-orderable} if it admits a strict total ordering \(<\) such that \(g < h\) implies \(fg < fh\) for all \(f, g, h \in G\).

\begin{proposition}
\label{prop : equivalent defs LO}
The following are equivalent:
\begin{enumerate-(1)}
\item
\(G\) is left-orderable.
\item
There is \(P\subseteq G\) such that
	\begin{enumerate-(a)}
	\item
	\(P \cdot P\subseteq G\);
	\item
	\(P\sqcup P^{-1} = G\setminus \{id\}\).
	\end{enumerate-(a)}
\item
\label{item : equivalent defs LO 3}
There is a totally ordered set \((\Upomega, <)\) such that \(G\hookrightarrow \Aut(\Upomega,<)\).
\end{enumerate-(1)}
\end{proposition}

The subset $P$ in $(2)$ above is referred to as a \emph{positive cone}.  Every left-ordering $<$ of $G$ determines a positive cone $P_{<}= \{ g \in G : g> id\}$.  The identification of left-orderings with the corresponding positive cones allows us to define the space of left-orderings as follows.  Equip $\{0,1\}$ with the discrete topology, $\{0,1\}^G$ with the product topology, and set 
\[ \LO(G) = \{ P \subset G : P \mbox{ is a positive cone }\} \subset \{0,1\}^G,
\]
equipped with the subspace topology.  Note that the subbasic open sets in $\LO(G)$ are the sets of the form $U_g = \{ P : g \in P \}$, where $ g \in G\setminus \{ id \}$.  One can easily check that $\LO(G)$ is a closed subset of \(\{0,1\}^G\), hence a compact Polish space. We regard \(\LO(G)\) as a Polish \(G\)-space in the following precise sense. There is a $G$-action by homeomorphisms on $\LO(G)$, given by $g \cdot P = gPg^{-1}$. As mentioned in the introduction, we denote by \(E_\mathsf{lo}(G)\) the orbit equivalence relation on \(\LO(G)\) induced by the conjugacy \(G\)-action.

\subsection{Smoothness of \(E_\mathsf{lo}(G)\) and relatively convex subgroups} 
Let \(G\) be a group equipped with a fixed left-ordering $<$. A subgroup $C$ of $G$ is \emph{convex relative to $<$} if whenever $g, h\in C$ and $f \in G$ with $g<f<h$, then $f \in C$.  A subgroup  \(C\subseteq G\) is \emph{left-relatively
 convex} in \(G\) (or \emph{relatively convex} in \(G\) for short) if \(C\) is convex relative to some left ordering of \(G\). 

Suppose that \(E\), \(F\) are countable Borel equivalence relations on the Polish spaces \(X\) and \(Y\), respectively.
Then the Borel map \(\varphi \colon X \to Y\)
is a \emph{Borel homomorphism}
from \(E\) to \(F\)
if \(x\mathbin{E} y \implies \varphi(x)\mathbin{F}\varphi(y)\). If the Borel homomorphism \(\varphi \colon X \to Y\)
from \(E\)
to \(F\)
is countable-to-one, then we say that \(\varphi\) is a \emph{weak Borel reduction} (in symbols, \(E\leq^w_{B} F\)). As pointed out by the kind referee we can establish the following fact:

 \begin{proposition}
 \label{prop : convex}
     If \(C\) is relatively convex in \(G\), then \(E_\mathsf{lo}(C) \leq^w_B E_\mathsf{lo}(G)\).
     Thus, any property of \(E_\mathsf{lo}(G)\) which is downward closed under weak Borel reductions (such as
smoothness, hyperfiniteness, \(\alpha\)-amenability, treeability, etc.), passes to \(E_\mathsf{lo}(C)\) for every relatively
convex subgroup \(C \leq G\).
 \end{proposition}
 \begin{proof}
Suppose that \(C\) is 
relatively convex in \(G\).
We let \(G\) act on the quotient \(G/C\) by conjugation.
A subgroup \(C\) is relatively
convex in a left-orderable group \(G\) if and only if there is a \(G\)-invariant linear order on \(G/C\) (e.g., see~\cite{AntDicSun,Ber90}), so fix such a linear order \(<_{G/C}\) . Then for every left-order \(P\in \LO(C)\), define the left-order \(\bar P \in \LO(G)\) lexicographically
by declaring
\[
g\in \bar P\qquad \iff \qquad C <_{G/C} gC\quad\text{or}\quad(g\in C\text{ and }g\in P).
\]
I.e., \(\bar P\) is the union of \(P\) and all positive cosets. Then the function \(\LO(C) \to \LO(G), P \to \bar P\) is one-to-one and 
\(C\)-equivariant. In particular, it is a one-to-one weak Borel reduction from
\(E_\mathsf{lo}(C)\) to \(E_\mathsf{lo}(G)\).
 \end{proof}



We leverage the dichotomy established in Proposition~\ref{prop : smooth acc points} to reprove a special case of Proposition~\ref{prop : convex}: that the non-smoothness of \(E_\mathsf{lo}(G)\) is detected by relatively convex subgroups. First we need the following observation.

\begin{proposition}
\label{prop : convex smooth}  Suppose that $G$ is left-orderable, \(C\leq G\) is relatively convex, and \(Q\) is condensed in \(\LO(C)\). If \(P\in \LO(G)\) satisfies \(P\cap C = Q\), then \(P\) is condensed in \(\LO(G)\). 
\end{proposition}
\begin{proof}
    Fix a relatively convex \(C\leq G\) and a positive cone \(Q\in \LO(C)\) that is an accumulation point in \(\Orb_C(Q)\). Also let \(P\in\LO(G)\) with \(P\cap C =Q\).  We will need the following:
\begin{proof}[Claim]
If \(c\in C\) and \(P\in \bigcap^n_{i=1} U_{g_i}\) for \(g_i\in G\setminus C\), then \(c P c^{-1} \in \bigcap^n_{i=1} U_{g_i}\).
\renewcommand{\qedsymbol}{}
\end{proof}
\begin{proof}[Proof of the Claim]
Assume that $c \in P$. 
 Since \(C\) is convex with respect to \(<_P\), we have \(c<_Pg_i\) for all \(i=1,\dotsc, n\). So, for \(i=1,\dotsc, n\) we have  \(c^{-1}g_i \in P\) and therefore \(c^{-1}g_i c \in P\). We obtain \(g_i\in c Pc^{-1}\) for all \(i\leq n\), therefore
\(c Pc^{-1}\in \bigcap^n_{i=1} U_{g_i}\).

Next suppose that $c \in P^{-1}$.  Then $g_i^{-1} <_P c$ for all $i = 1, \ldots, n$ since $C$ is convex. Therefore $g_ic \in P$ and thus $c^{-1} g_i c \in P$ and we conclude as in the previous case, completing the proof of the claim.
\end{proof}

Now let \(P\in \bigcap^n_{i=1} U_{c_i}\cap\bigcap^m_{j=1} U_{g_j}\) for some $c_1, \ldots, c_n \in C$ and $g_1, \ldots, g_m \in G \setminus C$. Since \(P\cap C \in \Orb_C(P\cap C)'\) there exists \(c\in C\) such that 
\[
c(P\cap C) c^{-1} \neq P\cap C
\quad\text{ and }\quad
c(P\cap C) c^{-1} \in \bigcap^n_{i=1} U_{c_i}.
\] 
Then
\(cPc^{-1} = c(P\cap C)c^{-1} \cup c (P\setminus C)c^{-1}\neq P\), and
\(
cPc^{-1} \in  \bigcap^n_{i=1} U_{c_i} \cap \bigcap^m_{j=1} U_{g_i}\) by the previous claim.
\end{proof}

\begin{corollary}
    \label{cor : rel convex}
For a left-orderable group \(G\), the following are equivalent:

\begin{enumerate-(1)}
\item
\label{item : rel convex 1}
\(E_\mathsf{lo}(G)\) is smooth.
\item
\label{item : rel convex 2}
For every relatively convex \(C\leq G\) the conjugacy orbit equivalence relation \(E_\mathsf{lo}(C)\) is smooth.
\end{enumerate-(1)}
\end{corollary}
\begin{proof}
    The only non-trivial implication is  \ref{item : rel convex 1} $\implies$ \ref{item : rel convex 2}. Fix a relatively convex \(C\leq G\) such that \(E_\mathsf{lo}(C)\) is not smooth. It follows from Proposition~\ref{prop : smooth acc points} that there is a positive cone \(Q\in \LO(C)\) that is an accumulation point in \(\Orb_C(Q)\).
Since \(C\) is relatively convex in \(G\) we can find some positive cone \(P\in\LO(G)\) such that \(Q = P\cap C\).
Proposition~\ref{prop : convex smooth} yields that \(P\) is condensed in \(\LO(G)\), hence \(E_\mathsf{lo}(G)\) is not smooth.
\end{proof}

\begin{remark}
Note that the condition on relatively convex subgroups in Corollary~\ref{cor : rel convex}\ref{item : rel convex 2} cannot be replaced with a condition on \emph{proper} relatively convex subgroups, as the example below shows (see also \cite{CalCla22}).
\end{remark}

Consider the infinitely generated group \footnote{This example also appears in \cite[Example~2.10]{CalCla22}, where there is a typo in the group presentation which is corrected here. We acknowledge Meng Che ``Turbo'' Ho for finding the typo and suggesting how to fix it.}
    \[ H_\infty = \langle x_1, x_2, \ldots \mid x_i x_{i-1}x_i^{-1} = x_{i-1}^{-1} \text{ for $1<i$ and } x_i x_j = x_j x_i \text{ for $|i -j|>1$} \rangle.
\]

Then for every left-ordering of $H_{\infty}$, one can show the convex subgroups are precisely the finitely generated subgroups of the form $H_j = \langle x_1, x_2, \dots ,x_j \rangle$ where $j \geq 1$.  This follows from first observing that every element $H_{j}$ can be represented by a word of the form 
\[ x_1^{a_1}x_2^{a_2} \dots x_{j}^{a_{j}}
\]
where $a_i \in \mathbb{Z}$, by using repeated applications of $x_{j} x_{j-1} = x_{j-1}^{-1} x_{j}$ and $x_{j} x_{i} = x_{i} x_{j}$ for all $i<j-1$ to shuffle all occurrences of $x_{j}$ to the right hand side of any representative word.  By writing every element of $H_{\infty}$ in this form, it is straightforward to check that $H_{j}$ is convex relative to every left-ordering of $H_{\infty}$.  Moreover, there are no other relatively convex subgroups aside from the subgroups $H_j$.  For if $C$ were such a subgroup, there would exists $j$ such that $H_j \leq C \leq H_{j+1}$.  But then $C$ should descend to a convex subgroup of $H_{j+1}/H_j \cong \mathbb{Z}$ under the quotient map, which is only possible if $C = H_j$ or $C=H_{j+1}$ since there are no proper, nontrivial convex subgroups in $\mathbb{Z}$.

Now one observes that the left-orders of \(H_\infty\) are in bijective correspondence with sequences \((\epsilon_i) \in \{0, 1\}^\mathbb{N}\)
that encode the signs of the generators: for example we can set
\(x_i > id\) if and only if \(\epsilon_i = 1\). It is not hard to see that the conjugacy
action of \(H_\infty\) on the set \(\LO(H_\infty)\) yields an action of \(H_\infty\) on \(\{0, 1\}^\mathbb{N}\) given by \(x_j \cdot (\epsilon_i)\)
is the same as \((\epsilon_i)\) in every entry except the \((j - 1)\)-th position, which has been
changed.
Two left-orderings of \(H_\infty\) are in the same orbit if and only if their
corresponding sequences in \(\{0, 1\}^\mathbb{N}\)
are eventually equal.

Thus every relatively convex proper subgroup $C \leq H_{\infty}$ is a Tararin group\footnote{Recall that a left-orderable group is \emph{Tararin} if it admits exactly finitely many left-orders.}, so $\LO(C)$ is finite, and yet $E_{\mathsf{lo}}(H_{\infty})$ is not smooth.

\subsection{Non-smoothness and isolated points}

Recall that a positive cone $P$ determines a \emph{Conradian} left-ordering of $G$ if $g, h \in P$ implies $g^{-1}hg^2 \in P$ for all $g, h \in G$~\cite{Nav10}.  Given a positive cone $P \in \LO(G)$, the \emph{Conradian soul} of $<_P$ is the (unique) subgroup $C \leq G$ that is maximal with respect to the conditions:
\begin{enumerate}
    \item $C$ is convex relative to the ordering $<_P$ of $G$, and
    \item $P \cap C$ determines a Conradian left-ordering of $C$.
\end{enumerate}

We recall the following theorem proved via different techniques in both \cite{Cla10} and \cite{Nav10}.

\begin{theorem}
\label{nontrivialsoul} 
If the Conradian soul of $<_P$ is trivial, then $P$ is condensed.
\end{theorem}

Thus if $G$ admits a positive cone $P$ having trivial Conradian soul, then $E_{\mathsf{lo}}(G)$ is not smooth.

As a consequence of Theorem \ref{nontrivialsoul}, every isolated point in $\LO(G)$ (that is, $P \in \LO(G)$ such that there exist $g_1, \dots, g_n \in G \setminus \{ id \}$ with $\{P \} = \bigcup_{i=1}^n U_{g_i}$) must have nontrivial Conradian soul, as isolated points cannot be condensed.  In fact, Navas shows much more:

\begin{theorem}\cite[Proposition 4.9]{Nav10}
Suppose that $P$ is an isolated point and let $C \leq G$ its Conradian soul. Then $C$ is a Tararin group, so $\LO(C)= \{ Q_1, \ldots, Q_{2^k}\}$ for some $k >0$, moreover,  if $G$ is not a Tararin group then there exists $i \in \{ 1, \ldots, 2^k\}$ such that $(P \setminus C) \cup Q_i$ is a condensed point of $\LO(G)$.
\end{theorem}

As an immediate consequence, we apply Proposition \ref{prop : smooth acc points} and observe:

\begin{corollary}
Suppose that $G$ is not a Tararin group.  If $\LO(G)$ contains an isolated point, then $E_{\mathsf{lo}}(G)$ is not smooth.
\end{corollary}

\subsection{The free part of \(\LO(G)/G\)}

For a left-orderable group \(G\) denote by \(\LO((G))\) the \emph{free part of its conjugacy action}. That is, we set
\[
\LO((G)) = \{ P \in \LO (G): \forall g\neq 1 (g^{-1} Pg \neq P ) \}.
\]

Note that for any \(P\in \LO ((G))\), the orbit \(\Orb(P)\) is infinite. 

\begin{proposition}
\label{prop : free part}
If \(\LO((G))\neq \emptyset\), then \(E_\mathsf{lo}(G)\) is not smooth.
\end{proposition}
\begin{proof}
Suppose \(P\in \LO((G))\).  By Proposition~\ref{prop : smooth acc points} it suffices to show that \(P\) is condensed. Let \(P\in  \bigcap^n_{i=1} U_{g_i}\), which is a basic open neighborhood of \(P\). And assume that \(g_1<_P\dotsb<_Pg_n\) without loss of generality. Then,  we claim that
\[
 g_1^{-1} g_i g_1 \in P
\]
for \(i=1,\dotsc, n\).  For \(i=1\), it follows from the assumption that \(P\in U_{g_1}\).  For \(i\geq 2\), \(g_1<_P g_i\) implies that \(1 <_Pg_1^{-1} g_i\), whence \(1 <_Pg_1^{-1} g_ig_1\) because \(P\) is a semigroup.
Therefore,  for \(i=1,\dotsc, n\), we have 
\[
  g_i  \in g_1Pg_1^{-1}.
\]
This shows that \(g_1 P g_1^{-1} \in \Orb(P)\cap \bigcap^n_{i=1} U_{g_i} \).
Since \(P\in \LO((G))\), we conclude that \(g_1^{-1} P g_1 \neq P\), therefore \(P\) is condensed.
\end{proof}

\subsection{Baumslag-Solitar groups}
Fix an integer \(n\).  The Baumslag–Solitar group \(\BS(1, n)\) is given by the presentation \(\langle a, b \mid bab^{-1} = a^n\rangle\).  There is an injective homomorphism \(\rho\colon \BS(1,n) \to \Homeo_{+}(\mathbb{R})\) defined by setting
\begin{align*}
\rho(a)(x) &= x+1,\\
\rho(b)(x) &= nx.
\end{align*}

The following construction of left-orderings on \(\BS(1,n)\) is due to Smirnov~\cite{Smi66}.  For any \(\alpha\in \mathbb{R}\setminus \mathbb{Q}\) we can define a corresponding \(P_\alpha \in \LO\big( \BS(1, n) \big)\) by declaring
\[
g\in P_\alpha \iff \rho(g)(\alpha) > \alpha.
\]
Note that the map \(\mathbb{R}\setminus \mathbb{Q} \to \LO(\BS(1,n)), \alpha\mapsto P_\alpha\) is injective. In fact,
for different irrational numbers \(\alpha<\beta\), we can choose some \(g \in \BS(1, n)\) such that \[\rho(g) = n^r x + \frac{s}{n^t}\] with \(r>0\), and having fixed point \(q = \frac{s}{n^t(1-n^r)}\) strictly between \(\alpha\) and \(\beta\). This choice is always possible because the range of \(\rho\) consists of precisely those functions of the form \(f(x) = n^r x + \frac{s}{n^t}\) with \(r,s,t\in \mathbb{Z}\). Moreover, we can always choose \(t\) and \(r\) so that the denominator
\(D=n^t(1-n^r)\) satisfies  \(1/|D| < \beta - \alpha\), therefore the interval \((\alpha, \beta)\) must contain a point of the form \(\frac{m}{D}\), for some \(m\in\mathbb{Z}\). Then, since \(r>0\), we have $\rho(g)(\alpha) < \alpha$ for all $\alpha < q$, and $\rho(g)(\beta) > \beta$ for all $\beta >q$. This means that \(g\in P_\beta \setminus P_\alpha\), showing that the function \(\alpha\mapsto P_\alpha\) is injective.

It is well known that the conjugacy action \(\BS(1, n)\curvearrowright\LO(\BS(1,n))\) is not generically ergodic, however with our new technique we can easily prove the following:

\begin{corollary}
\label{cor : BS(1,n)}
For \(n> 1\)  and \(G=\BS(1, n)\),  the conjugacy relation \(E_\mathsf{lo}(G)\) is not smooth.  
\end{corollary}

\begin{proof}
By Proposition~\ref{prop : free part} is suffices to prove that for any \(\alpha\in \mathbb{R}\setminus \mathbb{Q}\), the positive cone \(P_\alpha\) belongs the free part of the conjugacy action. To see this, assume that \(hP_{\alpha}h^{-1}= P_\alpha\).  One checks that $hP_{\alpha}h^{-1} = P_{\rho(h)(\alpha)}$. Therefore, we have \(P_{\rho(h)(\alpha)}= P_\alpha\) and, since the map \(\alpha\mapsto P_\alpha\) is injective,  it yields that \(\rho(h)(\alpha)=\alpha\). However, for every \(h\neq 1\), the order-preserving homeomorphism \(\rho(h)\) has only rational fixed points. Therefore,  it must hold that \(h\) is the group identity as desired.
\end{proof}

It is worth pointing out that Corollary~\ref{cor : BS(1,n)} also follows from Proposition~\ref{prop : smooth acc points} and the work of Rivas and Tessera~\cite[Proposition~2.12]{RivTes}. However, our analysis of
the free part of \(\BS(1,2) \curvearrowright \LO\big(\BS(1,2)\big)\) allows us to further settle the Borel complexity of \(E_\mathsf{lo}(\BS(1,2))\).
Recall that an equivalence relation \(E\) is \emph{hyperfinite} if it is the union of an increasing sequence of finite Borel equivalence relations.

\begin{corollary}
\label{cor : BS(1,2)}
    \(E_\mathsf{lo}(\BS(1,2))\) is hyperfinite.
\end{corollary}
\begin{proof}
Let \(G = \BS(1,2)\).
     and let \(Y=\{P_\alpha\mid \alpha\in \mathbb{R}\setminus \mathbb{Q}\}\) be the set of Smirnov's left-orders.
     Rivas~\cite[Theorem~4.2]{Riv10} establishes that \(\LO(G) \setminus Y\) is countable, therefore \(Y\) is Borel. Moreover, \(Y\) is closed under conjugation.
    Therefore \(Y\) is a free standard Borel \(G\)-space with the standard Borel structure induced by \(\LO(G)\).
    It follows that \(E_\mathsf{lo}(G) \sim_B E^Y_{G}\), and the latter equivalence relation is hyperfinite by \cite[Corollary~7.4]{CJMSTD}.
\end{proof}

To the best of our knowledge Corollary~\ref{cor : BS(1,2)} provides the first example of finitely generated left-orderable group, whose conjugacy equivalence relation is not smooth, yet hyperfinite.

\subsection{Thompson's group \(F\)}

\label{sec : Thompson}

Thompson's group $F$ may be defined by the presentation 
\[ F = \langle a, b \mid [ab^{-1}, a^{-1}ba], [ab^{-1}, a^{-2}ba^2]\rangle.
\]
There is an injective homomorphism $\rho \colon F \rightarrow 
\mathrm{PL}_+([0,1])$ whose image consists of all piecewise linear homeomorphisms of $[0,1]$ having dyadic rational breakpoints, and whose linear segments have slopes that are integral powers of two.

Given an interval $I = [\frac{p}{2^q}, \frac{p+1}{2^q}] \subset [0,1]$, we can define functions $b_I^+, b_I^{-}\colon [0,1] \rightarrow [0,1]$ that lie in the image of $\rho$ and whose support is equal to $I$, as follows.  First, the function $b_I^+$ is given by

 \begin{equation*}
    b_I^+(t) =
    \begin{cases*}
      t & if $0 \leq t \leq \frac{p}{2^q}$ \\
      2t-\frac{p}{2^q}       & if $\frac{p}{2^q} \leq t \leq \frac{p}{2^q} + \frac{1}{2^{q+2}}$\\
       t+\frac{1}{2^{q+2}}       & if $\frac{p}{2^q} + \frac{1}{2^{q+2}} \leq t \leq \frac{p}{2^q} + \frac{1}{2^{q+1}}$\\
          \frac{1}{2}t + \frac{p+1}{q^{q+1}}       & if $\frac{p}{2^q} + \frac{1}{2^{q+1}} \leq t \leq \frac{p+1}{2^q}$\\
          t & if $\frac{p+1}{2^q} \leq t \leq 1$ 
    \end{cases*}
  \end{equation*}
It is clear from this description that $b_I^+$ lies in the image of $\rho$.  On the interval $I$, the graph of $b_I^+$ appears as in Figure \ref{figure}(A).  We can analogously define $b_I^-$, which is the identity outside of $I$ and whose graph appears as in Figure \ref{figure}(B).

\begin{figure}
  \centering
  \subfloat[$b_I^+$]{\includegraphics[scale=0.5]{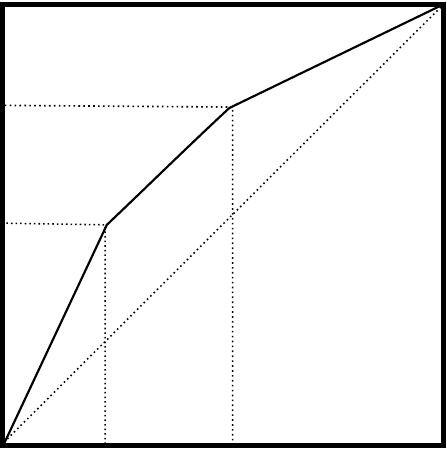}}
  \hspace{3em}
  \subfloat[$b_I^-$]{\includegraphics[scale=0.5]{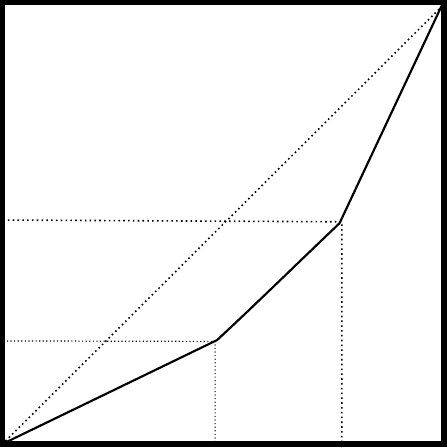}}
  \caption{The graphs of the functions $b_I^+$ and $b_I^+$ on the interval $I$.}
  \label{figure}
\end{figure}

\begin{proposition}
\label{prop:goodfunction}
    Let $S \subset [0,1]$ be finite, and choose $x, y \in [0,1] \setminus S$ with $x \neq y$.  Then there exists $g \in F$ such that $\rho(g)(s) = s$ for all $s \in S$, $\rho(g)(x)>x$, and $\rho(g)(y)<y$.
\end{proposition}
\begin{proof}
    As the dyadic rational numbers are dense, we may choose disjoint intervals $I, J$ with dyadic rational endpoints, satisfying $I \cap S = J \cap S = \emptyset$, with $x \in I$ and $y \in J$. Now set $f = b_I^+ \circ b_J^-$, then $f$ satisfies $f(s) = s$ for all $s \in S$, and $f(x)>x$ while $f(y)<y$.  Moreover, $f$ is in the image of $\rho$, so the proposition follows.
\end{proof}

Fix an enumeration $e\colon \mathbb{N} \rightarrow \mathbb{Q} \cap (0,1)$, writing $e(i) = r_i$.  Every enumeration of $\mathbb{Q}$ can be used to define a positive cone $P_e \subset F$ in the usual way: Given $g \in F$, let $r_i$ denote the first rational number in the enumeration satisfying $\rho(g)(r_i) \neq r_i$.  Then declare $g \in P_e$ if and only if $\rho(g)(r_i) > r_i$.

Theorem~\ref{thm : F} now follows from the following:

\begin{proposition}
For every enumeration $e\colon \mathbb{N} \rightarrow \mathbb{Q} \cap (0,1)$, we have $P_e \in \mathrm{LO}((F))$.
\end{proposition}
\begin{proof}
    Let $h \in F\setminus\{id\}$ and suppose that $\rho(h)(r_i) = r_i$ for all $i < N$, while $\rho(h)(r_N) \neq r_N$.  Set $S = \{ r_0, \ldots, r_{N-1}\}$, $x = r_N$ and $y=\rho(h)(r_N)$.  Apply Proposition \ref{prop:goodfunction} to arrive at $g \in F$ with $\rho(g)(r_i) = r_i$ for all $i<N$ and $\rho(g)(r_N) >r_N$, so that $g \in P_e$.  On the other hand, $\rho(h^{-1}gh)(r_i) = r_i$ for all $i<N$, while $\rho(g)(\rho(h)(r_N)) < \rho(h)(r_N)$ holds by our choice of $g$, which is equivalent to $\rho(h^{-1}gh)(r_N)<r_N$.  Thus $h^{-1}gh \notin P_e$, meaning $g \notin hP_eh^{-1}$.  Thus $P_e \neq hP_eh^{-1}$.
\end{proof}
From this, we conclude:

\begin{corollary}$E_{\mathrm{lo}}(F)$ is not smooth.\end{corollary}

\section{Open problems}

The results of this paper and our previous work produce several techniques to handle the problem whether \(E_\mathsf{lo}(G)\) is smooth. The next degree of Borel complexity is that of hyperfinite equivalence relations that are not smooth. 
Since every Borel equivalence relation with countable classes is hyperfinite on a comeager set, it is generally very difficult to show that a given equivalence relation is hyperfinite. So we pose the following open questions, which are likely to require new techniques.

In a previous draft of this paper we asked whether \(E_{\mathrm{lo}}\big(\BS(1,n)\big)\) is hyperfinite for \(n>2\). This question was addressed by Ho, Le, and Rossegger~\cite{HoLeRos},
who gave an alternative proof of Theorem~\ref{thm : BS} and answered our question affirmatively. 

Therefore, it is natural to ask if our methods can be used to analyze other solvable groups, and more generally:

\begin{question}
    What is the Borel complexity of \(E_\mathsf{lo}(G)\), for \(G\) abelian-by-abelian?
\end{question}

Regarding Thompson's group \(F\) and the complexity of \(E_\mathsf{lo}(F)\), very little is known beyond Theorem~\ref{thm : F}.
In particular, we ask the following question:
\begin{question}
\label{q : Thompson F}
Is $E_{\mathrm{lo}}(F)$ hyperfinite?
\end{question}

Question~\ref{q : Thompson F} is related to
two famous problems: whether Thompson's group \(F\) is amenable, and whether every countable Borel equivalence relation induced by the action of an amenable group is hyperfinite, a long-standing open question posed by Benjamin Weiss. A negative answer to Question~~\ref{q : Thompson F} would imply that the amenability of \(F\) and a positive answer to Weiss question are mutually exclusive.

\bibliographystyle{alpha}
\newcommand{\etalchar}[1]{$^{#1}$}

\end{document}